\documentclass{amsart}
\usepackage{chngcntr}
\usepackage{apptools}
\usepackage{color}
\usepackage{amsthm,amssymb,verbatim}
\usepackage{graphicx}
\usepackage{enumerate}
\usepackage{chngcntr}
\usepackage{apptools}
\usepackage{color}
\usepackage{amsthm,amssymb,verbatim}
\usepackage{graphicx}
\usepackage{enumerate}

\newcommand{\Rr}{\mathcal R}

 \newcommand{\RR}{\mathbf{R}}  
 \newcommand{\BB}{\mathbf{B}}  

 \newcommand{\area}{\operatorname{area}}
 
 \newcommand{\Tan}{\operatorname{Tan}}

\newcommand{\ee}{\mathbf e}

\newcommand{\Hh}{\mathcal{H}}
\usepackage{amsthm}

\input epsf
\def\begfig {
\begin{figure}
\small }
\def\endfig {
\normalsize
\end{figure}
}

    \newtheorem{theorem}    {Theorem}   
    \newtheorem{lemma}      [theorem]       {Lemma}
    \newtheorem{corollary}  [theorem]     {Corollary}
    \newtheorem{proposition}       [theorem]       {Proposition}

    \newtheorem*{theorem*}{Theorem}
    \theoremstyle{definition}

    \theoremstyle{definition}
    \newtheorem{remark}   [theorem]       {Remark}

\title[Huisken's Monotonicity]{The Boundary Term in Huisken's Monotonicity Formula and 
  the Entropy of Translators }
\author{Brian White}
\address{Department of Mathematics\\ Stanford University\\ Stanford, CA 94305}
\date{4 April, 2022.  Revised 23 July, 2023.}
\usepackage{hyperref}
\usepackage{enumerate}
\usepackage[alphabetic, msc-links, backrefs]{amsrefs}
\begin{document}

\begin{abstract}
For a manifold-with-boundary moving by mean curvature flow, the entropy
 at a later time is bounded by the entropy at an earlier time plus a boundary term.  
 This paper controls the boundary term in a geometrically natural way.
 In particular, it shows (under mild hypotheses) 
 that the entropy of a compact translator
 is less than or equal to the entropy of the boundary plus the maximal cone density of the boundary.
 \end{abstract}

\maketitle

\newcommand{\pdf}[2]{\frac{\partial #1}{\partial #2}}
\newcommand{\pdt}[1]{\frac{\partial #1}{\partial t}}

\newcommand{\entropy}{\operatorname{entropy}}
\newcommand{\mcd}{\operatorname{mcd}}
\newcommand{\mdr}{\operatorname{mdr}}
\newcommand{\tM}{\tilde M}
\newcommand{\tF}{\tilde F}
\newcommand{\tmu}{\tilde \mu}

\section{Introduction}
For a closed surface moving by mean curvature flow in Euclidean space, Huisken's monotonicity formula~\cite{huisken-asymptotic}*{\S3}
implies that a certain weighted area decreases in time.  That in turn implies that
the entropy of the surface is a decreasing function of time.

For mean curvature flow of surfaces with boundary (where the motion of the boundary is prescribed), the entropy need not decrease, because the analog of Huisken's monotonicity formula includes a spacetime boundary integral.    In order to bound entropy at one time in terms of the entropy at an early time, it is necessary to control the boundary integral.

In this paper, we control the boundary integral in a 
geometrically natural way. 
(For the easier case of non-moving boundaries,
 see~\cite{white-mcf-boundary}*{Theorem~7.1}.)
 In particular, we show that it is bounded by the Gaussian area of the surface swept out
by a certain time-dependent rescaling of the the boundary.

As an application, we prove a simple, explicit bound for the entropy of a compact
translator with boundary, provided the boundary lies in a hyperplane or finite union
of hyperplanes perpendicular to the direction of translation.  
The surface can be of any dimension and codimension.  For example, as a special case, we have

\begin{theorem}\label{convex-intro-theorem}
Consider an $m$-dimensional compact surface $M$ in $\RR^{m+1}$ that translates with velocity $v\ee_{m+1}$ under mean curvature flow.  Suppose that $\partial M$ consists of $k$ components, each of which is the boundary of a convex region in a horizontal $m$-plane.  Then the entropy of $M$ is at most 
\[
   k \left( 1 + \frac{m\omega_m}{\omega_{m-1}} \right).
\]
\end{theorem}
Here $\omega_n$ denotes the volume of the unit ball in $\RR^n$.

\begin{corollary}
For such a surface $M$,
\begin{equation}\label{mdr-bound}
  \frac{\area(M\cap\BB(x,r))}{\omega_m r^m}  \le C_m k
\end{equation}
for all balls $\BB(x,r)$, where $C_m$ depends only on $m$.
\end{corollary}

The corollary follows from the theorem because density ratios 
(for any surface) are bounded by 
a constant times the entropy; see~\eqref{entropy-mdr} below.
The bound~\eqref{mdr-bound} plays a key role in constructing
families of complete, non-rotationally invariant translating
 annuli in~$\RR^3$~\cite{annuloids}.
 
See Section~\ref{vertical-section} for the analogous results for boundaries $\partial M$ 
that may also contain some vertical components.

We remark that Ilmanen's elliptic regularization~\cite{ilmanen}
 obtains very general mean curvature flows by taking
suitable limits of sequences of translators.  The bounds in this paper imply entropy bounds for
such limits.

We now describe the general result for translators.

If $\Sigma$ is a $(m-1)$-dimensional submanifold of Euclidean space, we let
\[
C(\Sigma) := \{rx: x\in \Sigma, \, r> 0\}
\]
be the cone over $\Sigma$ with vertex at the origin.  The {\bf density} of the cone is 
\[
  \Theta(C(\Sigma)) := \frac{\area(C(\Sigma)\cap \BB(0,r))}{\omega_m r^m},
\]
where $\omega_m$ is the volume of the unit ball in $\RR^n$.   (Note that the right hand side does not depend on $r$.)  Here, area should be counted with multiplicity.    

We define the {\bf maximal cone density} of $\Sigma$ to be 
\[
  \mcd(\Sigma): = \sup_{v\in \RR^n} \Theta(C(\Sigma+v)).
\]

\begin{theorem}\label{general-intro-theorem}
Suppose that $M$ is a compact $m$-dimensional manifold in $\RR^n$ that translates
with velocity $v\ee_n$ under mean curvature flow.  Suppose that $\partial M$ lies in a horizontal hyperplane.
Then
\[
   \entropy(M) 
   \le
    \entropy(\partial M) + \mcd(\partial M).
\]
More generally, if $\partial M=\cup_i\Sigma_i$, where each $\Sigma_i$ lies in a horizontal hyperplane $P_i$, then
\[
 \entropy(M) \le \sum_i \left( \entropy(\Sigma_i) + \mcd(\Sigma_i) \right).
\]
\end{theorem}

Recall that if $S$ is a $d$-dimensional embedded submanifold of $\RR^n$,
 then 
 the {\bf maximal density ratio} of $S$ is 
\[
 \mdr(S) := \sup_{x\in\RR^n, \, r>0} \frac{\Hh^d(S\cap \BB(x,r))}{\omega_d r^d}.
\]
(If $S$ is an immersed variety, then the $d$-dimensional area should be counted
with multiplicity.)
Entropy and maximal density ratio are closely related: 
\begin{equation}\label{entropy-mdr}
   \entropy(S) \le \mdr(S) \le c_d \entropy(S),
\end{equation}
where $d=\dim(S)$.  (See~\cite{white-mcf-boundary}*{Theorem~9.1}.)
Thus as a consequence of Theorem~\ref{general-intro-theorem}, we have

\begin{corollary}\label{intro-density-corollary}
If $M$ is compact, $m$-dimensional translator in $\RR^n$ with velocity $v\ee_n$ and if $\partial M$
lies in a horizontal hyperplane, then
\[
 \entropy(M) \le   \mdr(\partial M) + \mcd(\partial M),
\]
and
\[
  \mdr(M) \le c_m (  \mdr(\partial M) + \mcd(\partial M)).
\]
More generally, if $\partial M = \cup_i \Sigma_i$, where each $\Sigma_i$ lies
in a horizontal hyperplane, then 
\[
 \entropy(M) \le \sum_i\left( \mdr(\Sigma_i) + \mcd(\Sigma_i) \right),
\]
and
\[
  \mdr(M) \le c_m \sum_i (  \mdr(\Sigma_i) + \mcd(\Sigma_i)).
\]
\end{corollary}

Theorem~\ref{convex-intro-theorem} follows from
 Corollary~\ref{intro-density-corollary} because if $\Sigma$ is the boundary of a convex region in a $m$-plane, then
\[
   \mdr(\Sigma)\le \frac{m\omega_m}{\omega_{m-1}}
\]
(see Proposition~\ref{mdr-proposition}), and
\[
  \mcd(\Sigma)=1.
\]

\begin{remark}
If $\Sigma$ is a closed curve in $\RR^n$,
 then, according to~\cite{EWW}*{Theorem~1.1},
\[
  \mcd(\Sigma) = \frac1{2\pi}(\operatorname{TotalCurvature}(\Sigma)).
\]
\end{remark}

\section{The Monotonicity Inequality}

Consider  a closed, compact $(m-1)$-dimensional manifold $\Sigma$, and let 
\[
   F: \Sigma\times (-\infty,0) \to \RR^n
\]
be a $1$-parameter family of embeddings of $\Sigma$.  Let
$\Gamma(t)$ be the image of $F(\cdot,t)$:
\[
  \Gamma(t)= F(\Sigma,t).
\]
For $q\in \Gamma(t)$, let $\dot\Gamma(q,t)$ be the normal velocity of
 $\Gamma(t)$ at $q$:
\[
\dot \Gamma(q,t) = \left( \pdt{} F(x,t) \right)^\perp, \, \text{where $F(x,t)=q$.}
\]
Here $(\cdot)^\perp$ denotes the component perpendicular to $\Tan(\Gamma(t),q)$.

We will use the following rescaling of $F$:
\[
 \tF(x,t) = \frac{F(x,t)}{|t|^{1/2}}.
\]
We let $d\mu$ and $d\tmu$ be the $(m-1)$-dimensional volume measure on $\Sigma$
corresponding to $\Hh^{m-1}$ on $\Gamma(t)= F(\Sigma,t)$ and on $\tF(\Sigma,t)$.
Thus $d\mu = |t|^{(m-1)/2} \,d\tmu$.

Let $t\in (-\infty, 0)\mapsto M(t)$ be a one-parameter family of $m$-dimensional manifolds-with-boundary such that
\[
   \partial M(t) = \Gamma(t)
\]
and such that $M(t)$ moves by mean curvature flow: for each $x\in M(t)$, the normal
velocity at $(x,t)$ is the mean curvature vector $H(x,t)$ of $M(t)$ at $x$.
For $x\in \Gamma(t)$, we let $\nu_M(x,t)$ be the unit vector that it tangent to $M(t)$, normal to $\Gamma(t)$, and that points out from $M(t)$.
We let $\tM(t)$ be the rescaled surface
\begin{equation}\label{rescaled}
   \tM(t) = \frac{M(t)}{|t|^{1/2}}. 
\end{equation}

(The reader may wonder why we use the flow~\eqref{rescaled} rather than the
standard renormalized flow 
\begin{equation}
   \tau \mapsto \tM(-e^{-\tau}).
\end{equation}
The latter flow has a nicer equation of motion, but in this paper, there is no advantage 
in changing the time variable.)

More generally, $M(\cdot)$ can be a Brakke flow with
 boundary $\Gamma(\cdot)$.  (See~\cite{white-mcf-boundary}*{Definition~ 8.1}.)
   In this case, the vector $\nu(x,t)$ is a vector of length 
  $\le 1$ that is perpendicular to $\Gamma(t)$ at $x$.  
If we think of $M(t)$ as a non-equilibrium soap film, then $-\nu(x,t)$ is the force per
unit $(m-1)$-dimensional measure that the soap film exerts on the
 boundary $\Gamma(t)$ at $x$.

Define
\begin{align*}
\rho_m(x,t) &= \frac1{(4 \pi |t|)^{m/2}} \exp\left( -\frac{|x|^2}{4|t|} \right), \\
\Phi_m(x) &= \rho(x,-1) = (4\pi)^{-m/2} \exp\left( -\frac14|x|^2 \right).
\end{align*}
Thus
\[
 \rho_m(x,t) = |t|^{-m/2} \Phi_m(x/|t|^{1/2}).
\]
We will sometimes write $\rho$ and $\Phi$ for $\rho_m$ and $\Phi_m$ when the $m$ is clear from the context.

If $S$ is an $m$-dimensional submanifold of $\RR^n$, we define its {\bf $\Phi$-area} to be
\[
  \Phi_m[S] :=\int_S \Phi_m \,d\Hh^m.
\]
The {\bf entropy} of $S$ is the supremum of $\Phi_m[S']$ among all surfaces $S'$ 
obtained from $S$ by translation and dilation.

Note that
\begin{align*}
\int_{x\in M(t)}\rho_m(x,t)\,d\Hh^mx
&=
\int_{x\in M(t)}\Phi_m(x/|t|^{1/2}) |t|^{-m/2}\,d\Hh^mx \\
&=\int_{y\in \tM(t)} \Phi_m(y) \, d\Hh^my \\
&=\Phi_m[\tM(t)].
\end{align*}

We now use Huisken's Monotonicity Inequality, modified for surfaces with boundary: 
see~\cite{white-mcf-boundary}*{Theorem~18.3}.  
(Note: the terms $K$ and $A$ in that theorem are $0$ here because the ambient space
is Euclidean.)  The monotonicity theorem states that for $a<b<0$,
\begin{equation}
\begin{aligned}
\Phi_m[\tM(b)] - \Phi_m[\tM(a)]
&\le
\int_{t=a}^b 
\int_{\Gamma(t)} 
\nu_M\cdot \left(  \dot\Gamma - \frac{\nabla \rho}\rho \right)\rho \, d\Hh^{m-1} \,dt
\\
&=
\int_{t=a}^b 
\int_{q\in \Gamma(t)} 
\nu_M\cdot \left( \dot\Gamma + \frac{q}{2t} \right)\rho \, d\Hh^{m-1} \,dt.
\end{aligned}
\end{equation}
(Here $\rho=\rho_m$.)
We can express this last quantity, $Q$, in terms of $F$:
\begin{align*}
Q
&=
\int_{t=a}^b 
\int_{\Sigma} 
\nu_M\cdot 
\left(  \left(\pdt{F} \right)^\perp + \frac{F}{2t}  \right) \rho(F ,t) \, d\mu \,dt
\\
&=
\int_{t=a}^b 
\int_{\Sigma} 
\nu_M\cdot 
\left(  \pdt{F}   + \frac{F}{2t}  \right) \rho(F,t) \, d\mu \,dt
\\
&=
\int_{t=a}^b 
\int_{\Sigma} 
\nu_M\cdot 
 (-t)^{1/2} \pdt{} \left(\frac{F}{(-t)^{1/2}}\right) 
  |t|^{-m/2} \, \Phi\left( \frac{F}{|t|^{1/2}}\right) \, d\mu \,dt
\\
&=
\int_{t=a}^b 
\int_{\Sigma} 
\nu_M\cdot 
  \pdt{} \left(\frac{F}{|t|^{1/2}}\right) 
   \Phi\left( \frac{F}{|t|^{1/2}}\right)  |t|^{-(m-1)/2} \, d\mu \,dt
\\
&=
\int_{t=a}^b 
\int_{\Sigma} 
\nu_M\cdot 
  \pdt{\tF} \,
   \Phi(\tF) \, d\tmu \,dt
\end{align*}
Here we have used
\[
   \nu_M\cdot \left(\pdt{F} \right)^\perp = \nu_M \cdot \pdt{F}, 
\]
which is true since $\nu_M$ is perpendicular to $\Gamma(t)$.

Recall that $\nu_M$ is a unit vector perpendicular to $\Gamma(t)$ and therefore
also to $\tF(M,t)$.  (In the case of Brakke Flow, $\nu_M$ is a vector of length at most one that is perpendicular to $\Gamma(t)$.)  Therefore,
\[
    \nu_M\cdot \pdt{\tF}   \le \left| \left(\pdt{\tF} \right)^\perp \right|.
\]
Thus
\[
 Q \le A(\tF,a,b)
\]
where
\[
A(\tF,a,b) 
=
\int_{t=a}^b 
\int_{\Sigma} 
\left| \left( \pdt{} \tF \right)^\perp \right|
   \Phi\left( \tF(x,t) \right) \, d\tmu \,dt
\]
The expression for $A(\tF,a,b)$ has a simple 
geometric meaning: it is the $m$-dimensional $\Phi$-area
  swept out by the $\tF(\Sigma,t)$
from $t=a$ to $t=b$. 
In other words, $A(\tF,a,b)$ is the $\Phi$-area of the immersion
\[
   \tF: \Sigma\times [a,b] \to \RR^n.
\]

We have shown:

\begin{theorem}\label{sweep-out-theorem}
Suppose $t\in (-\infty,0)\mapsto M(t)$ is an $m$-dimensional Brakke flow
with boundary $\Gamma(\cdot)$, where $\Gamma(t)=F(\Sigma,t)$.
For each $T<0$, let $N(T)$ be the surface swept out by
\[
     \tilde \Gamma(t):= \frac{\Gamma(t)}{|t|^{1/2}}
\]
from $t=T$ to $0$:
\[
    N(T) = \cup_{T\le t< 0} \tilde \Gamma(t).
\]
In terms of the parametrization $F$, 
\[
    N(T) = \tF( \Sigma\times [T,0)).
\]
Then
\[
      \Phi_m[\tM(t)] + \Phi_m[N(t)]
\]
is a decreasing function of $t$.
\end{theorem}

\begin{proof}
We showed above that
\begin{align*}
\Phi[\tM(b)] - \Phi[\tM(a)]
&\le
\Phi[ F| \Sigma \times [a,b] ]  \\
&=
\Phi[F|\Sigma\times [a,0)] - \Phi[F|\Sigma\times [b,0)]  \\
&=
\Phi[N(a)] - \Phi[N(b)].
\end{align*}
\end{proof}

\begin{corollary}\label{mono-corollary}
If the areas of the $M(t)$ are bounded above, or, more generally, if 
\[
     \lim_{t\to -\infty}\Phi_m[\tM(t)]  = 0,
\]
then
\begin{align*}
\Phi_m[\tM(t)] 
&\le \Phi_m[\tF | \Sigma\times (-\infty,t] ]  \\
&\le \Phi_m[\tF | \Sigma\times (-\infty,0)]
\end{align*}
for all $t\in (-\infty,0)$.  
\end{corollary}

Because of the corollary, it is useful to have upper bounds
for $\Phi_m[\tF | \Sigma\times (-\infty,0)]$.  The next section gives
an upper bound in the case of boundaries that move by translation.

\section{Translators}

\begin{theorem}\label{main-calculation}
Suppose that $\Sigma$ is a compact, embedded
$(m-1)$-dimensional manifold in $\RR^{n-1}$, and suppose that
\begin{align*}
    &F: \Sigma\times \RR \to \RR^n, \\
    &F(x,t) = (x,0) + (a+t)v\ee_n.
 \end{align*}
Let $S=\tF(\Sigma\times (-\infty,0))$.  Then
\begin{align*}
\Phi_m[S]
&\le
 \operatorname{entropy}(\Sigma) + \Theta( \{rx: x\in \Sigma, \, r\ge 0\}) 
\\
&\le 
\entropy(\Sigma)  + \mcd(\Sigma).
\end{align*}
\end{theorem}

\begin{proof}
According to Theorem~\ref{gauss-slicing} below, for any $m$-dimensional submanifold $S$ of $\RR^n$, 
\begin{equation}\label{slicing-formula}
\Phi_m[S]
\le
\Phi_m[\Pi(S)] + \int_{y\in \RR} \Phi_{m-1}[S^y] \Phi_1(y) \, dy ,
\end{equation}
where  $\Phi_m[\Pi(S)]$ is the $\Phi$-area (counting multiplicity) of the projection of $S$ to the horizontal $(n-1)$-plane, and where $S^y$ is the horizontal slice
\[
  S^y:= \{x\in \RR^{n-1}: (x,y)\in S\}.
\]

In our case, 
\[
F(x,t) = (x, (a+t)v )
\]
so
\[
\tF(x,t) = (|t|^{-1/2}x, |t|^{-1/2}(a+t) v).
\]
Write $r=|t|^{-1/2}$, so $t= - r^{-2}$.  Thus
\begin{equation}\label{form}
  S = \{ (rx,  (ra-r^{-1}) v ): r>0, \, x\in \Sigma\}.
\end{equation}

From~\eqref{form}, we see that the projection $\Pi(S)$ of $S$ to the horizontal
plane is precisely $C(\Sigma)=\{rx: r>0, \, x\in \Sigma\}$, the cone over $\Sigma$:
\begin{equation}\label{cone-projection}
\Pi(S) = C(\Sigma).
\end{equation}

For each $y\in \RR$, let $\Rr(y)$ be the set of $r>0$ such that 
\[
    (r a - r^{-1}) v = y,
\]
and let $n(y)$ be the number of elements of $\Rr(y)$.
From~\eqref{form}, we see that
\[
S^y = \cup_{r\in \Rr(y)} r\Sigma.
\]
Thus
\begin{equation}\label{horizontal-slice}
\begin{aligned}
\Phi_m[S^y]
&=
\sum_{r\in \Rr(y)} \Phi_m[r\Sigma]  \\
&\le
\sum_{r\in \Rr(y)} \entropy \Sigma  \\
&=
n(y) (\entropy \Sigma).
\end{aligned}
\end{equation}

By~\eqref{slicing-formula}, \eqref{cone-projection}, and~\eqref{horizontal-slice},
\begin{align*}
\Phi_m[S]
&\le
\Phi_m[C(\Sigma)] + (\entropy \Sigma) \int n(y) \Phi_1(y)\,dy.
\end{align*}
The asserted inequality follows, because simple calculations (see Lemmas~\ref{calculus} and~\ref{cone-lemma} below) 
show that
\[
  \int n(y) \Phi_1(y)\,dy \le 1,
\]
and that, for any $m$-dimensional cone $C$ with vertex $0$, 
\[
  \Phi_m[C] = \Theta(C).
\]
\end{proof}

\begin{lemma}\label{calculus}
Let $n(y)$ be the number of $r>0$ such that $(ra-r^{-1})v=y$, i.e., such that
$avr^2 - yr + v=0$.  Then
\[
  \int n(y) \Phi_1(y)\,dy \le 1.
\]
\end{lemma}

\begin{proof}
The roots $r$ are given by
\[
   r = \frac{y \pm \sqrt{y^2 - 4av^2}}{2av}.
\]
If $a<0$, then exactly one of the roots is positive, so $n(y)=1$ for all $y$ and
therefore
\[
 \int n(y) \Phi_1(y)\,dy  = \int \Phi_1(y)\,dy = 1.
\]
Now suppose that $a>0$.  Let us also suppose that $v>0$. (The case $v<0$ is essentially the same.)  If $y>2\sqrt{a}v$, then there are two real roots, both positive.
If $|y|< 2\sqrt{a}v$, there are no real roots.  If $y< -2\sqrt{a}v$, there are two real roots, both negative.  Thus
\begin{align*}
\int n(y) \Phi_1(y)\,dy
&=
2\int_{2\sqrt{a}v}^\infty \Phi_1(y)\,dy  \\
&\le
2\int_0^\infty \Phi_1(y)\,dy \\
&= 1. 
\end{align*}
\end{proof}

\begin{lemma}\label{cone-lemma}
If $C$ is an $m$-dimensional cone with vertex at the origin, then
\[
  \Phi_m[C] = \Theta(C).
\]
\end{lemma}

\begin{proof}
The volume of the cone in $\BB(0,r+dr)\setminus \BB(0,r)$
is $d(\omega_m\Theta(C) r^m)$.
Thus
\begin{align*}
\Phi_m[C] 
&= \int_{r=0}^\infty \frac1{(4\pi)^{m/2}} e^{-r^2/4} d(\omega_m\Theta(C) r^m) \\
&= \Theta(C) \int_{r=0}^\infty \frac1{(4\pi)^{m/2}} e^{-r^2/4} d(\omega_m r^m). \\
&= \Theta(C) \int_{\RR^m} \Phi_m(x)\,dx \\
&= \Theta(C).
\end{align*}
\end{proof}

\begin{theorem}\label{main-translator-theorem}
Suppose that  $M$ is an $m$-dimensional compact surface in $\RR^n$ that translates with velocity $v\ee_n$ under mean curvature flow, and suppose that $\partial M$ lies in a horizontal hyperplane.   Then
\[
  \entropy(M) \le \mcd(\partial M) + (\entropy (\partial M)).
\]
More generally, if $\partial M=\cup_{i=1}^k \Gamma_i$, where each $\Gamma_i$ lies in  a horizontal plane $P_i$, then
\[
  \entropy(M) \le \sum_{i=1}^k(\mcd(\Gamma_i) + (\entropy (\Gamma_i)).
\]
\end{theorem}

\begin{proof}
Consider the MCF
\[
 M(t):=  M + (t+1)v\ee_n.
\] 
By Corollary~\ref{mono-corollary} and Theorem~\ref{main-calculation}, 
\[
 \Phi_m[\tM(t)] \le \mcd(\partial M) + (\entropy (\partial M))
\]
for all $t<0$.
Now $\tM(-1)= M(-1) = M$, so, in particular,
\[
  \Phi_m[M] \le \mcd(\partial M) + (\entropy (\partial M)).
\]
Now let $M'$ be any surface obtained from $M$ by translating and dilating.
Then, by the same argument,
\[
 \Phi_m[M'] \le \mcd(\partial M') + (\entropy(\partial M')).
\]
But $\mcd(\partial M')=\mcd(\partial M)$
 and $\entropy(\partial M')=\entropy(\partial M)$. 
Thus
\[
 \Phi_m[M'] \le \mcd(\partial M) + (\entropy(\partial M)).
\] 
Taking the supremum over all such $M'$ gives
\[
 \entropy(M) \le \mcd(\partial M) +  (\entropy(\partial M)).
\]
The assertion for the case $\partial M=\cup \Gamma_i$ is proved in the same way.
\end{proof}

\section{Slicing}

\begin{lemma}\label{slicing}
Let $S$ be a smooth, $m$-dimensional manifold (possibly with boundary) in $\RR^n$.
For $x\in \RR^{n-1}$ and $y\in \RR$, let
\begin{align*}
   S_x &: = \{y\in \RR: (x,y)\in M\}, \\
   S^y &= \{x\in \RR^{n-1}: (x,y)\in M\}.
\end{align*}
Let 
\begin{align*}
&\Pi: \RR^n\to \RR^{n-1}, \\
&\Pi(x,y) = x  \qquad(x\in \RR^{n-1}, \, y\in \RR).
\end{align*}
Suppose $f: \RR^{n-1}\to \RR$ and $g: \RR\to\RR$ are smooth, nonnegative functions.
Then
\begin{equation}\label{slicing-inequality}
\begin{aligned}
\int_{(x,y)\in S} f(x)g(y)\, d\Hh^m(x,y)  
&\le
\int_{x\in \Pi(S)} \left( \sum_{y\in S_x} g(y) \right) f(x)\,d\Hh^{m-1} x \\
&\quad+
\int_{y\in \RR} \left( \int_{x\in S^y} f(x) \,d\Hh^{m-1}x \right) g(y) \, dy.
\end{aligned}
\end{equation}
\end{lemma}

\begin{proof}
Let 
\[
   h: (x,y)\in \RR^{n-1}\times \RR \mapsto y
\]
be the height function.  
Then $\nabla_Sh$
is the projection of $\nabla h=\ee^n$ to $\Tan(S,\cdot)$.
Note that the Jacobian of the map $\Pi:S$ is $J = \sqrt{1 - |\nabla_Mh|^2}$, and thus
that 
\begin{equation}\label{triangle}
   J + |\nabla_Sh| \ge 1.
\end{equation}
By the area formula for $\Pi|S$,
\begin{equation}\label{area}
\begin{aligned}
\int_{x\in \Pi(S)} \left(\sum_{y\in S_x}  g(y) \right) f(x) \,d\Hh^{m-1}x 
&=
\int_{x\in \Pi(S)} \left(\sum_{y\in S_x} f(x) g(y) \right) \,d\Hh^{m-1}x  \\
&=
\int_{(x,y)\in S}  f(x)g(y) J(x,y) \,d\Hh^m(x,y).
\end{aligned}
\end{equation}
Likewise, by the coarea formula for $h|S$,
\begin{equation}\label{coarea}
\begin{aligned}
\int_{y\in \RR} \left(\int_{x\in S^y} f(x) \,d\Hh^{m-1}x\right) g(y) \,dy
&=
\int_{y\in \RR} \int_{x\in S^y} f(x)g(y) \,d\Hh^{m-1}x\,dy  \\
&=
\int_{(x,y)\in S} f(x)g(y)\, |\nabla_S h|\,d\Hh^m(x,y).
\end{aligned}
\end{equation}
Now add~\eqref{area} and~\eqref{coarea} and use the inequality~\eqref{triangle}.
\end{proof}

\begin{remark}\label{counting-remark}
Note that if $g\le 1$, then $\sum_{y\in S_x}g(y)$ is less than or equal to the number of points in $S_x$, which is the multiplicity of the projection $\Pi|S$.
Thus (in this case) the first integral on the right hand side in~\eqref{slicing-inequality} is bounded above by
the integral of $f$ over $\Pi(S)$, counting multiplicity.
\end{remark}

\begin{theorem}\label{gauss-slicing}
Let $S$ be an $m$-dimensional submanifold in $\RR^n$.
Then
\begin{align*}
\Phi_m[S] 
&\le \Phi_m[\Pi|S] + \int_{y\in \RR} \Phi_{m-1}[S^y] \Phi_1(y)\,dy   \\
\end{align*}
where $\Phi_m[\Pi|S]$ is the $\Phi_m$-area of $\Pi(S)$, counting multiplicity.
\end{theorem}

\begin{proof}
Note that for $(x,y)\in \RR^{m-1}\times \RR$,
\begin{align*}
\Phi_m(x,y) 
&= \frac1{(4\pi)^{m/2}} \exp\left( \frac{-|x|^2-|y|^2}{4} \right)   
\\
&=
\frac1{(4\pi)^{(m-1)/2}} \exp\left( \frac{-|x|}{4} \right)   
\frac1{(4\pi)^{1/2}} \exp\left( \frac{-|y|^2}{4} \right)
\\
&=
\Phi_{m-1}(x)\Phi_1(y).
\end{align*}
The assertion of the theorem follows immediately from Lemma~\ref{slicing}
(letting $f(x)=\Phi_{m-1}(x)$ and $g(y)=\Phi_1(y)$)
and Remark~\ref{counting-remark}.
\end{proof}

\section{The Maximal Density Ratio of a Convex Surface}

\begin{proposition}\label{mdr-proposition}
Let $U$ be a bounded, convex, open region in $\RR^{m+1}$.  Then
\[
  \mdr(\partial U)
   \le  \frac{{(m+1)}\omega_{m+1}}{\omega_m}.
\]
\end{proposition}

\begin{proof}
For $x\in \RR^m$, let $\pi(x)$ be the point in $\overline{U}$ closest to $x$.  Then
\[
  |\pi(x)-\pi(y)| \le |x-y|
\]
for all $x,y\in\RR^m$.  Let $Q = \partial \BB(x,r)\setminus U$.
Then $\pi$ is a distance-decreasing map from $Q$ onto $\BB(x,r)\cap\partial U$.
Thus
\[
\Hh^m(\partial \BB(x,r)) \ge \Hh^m(Q) \ge \Hh^m(\BB(x,r)\cap \partial U),
\]
so
\[
\frac{\Hh^m(\BB^m(x,r)\cap\partial U)}{\omega_mr^m}
\le
\frac{\Hh^m(\partial \BB(x,r))}{\omega_mr^m}
 = \frac{(m+1)\omega_{m+1}}{\omega_m}.
 \]
\end{proof}

\section{Boundaries with Vertical Pieces}\label{vertical-section}

Now suppose that $t\mapsto M(t)$ is a mean curvature flow of $m$-dimensional manifolds in $\RR^n$
and that the boundary of $M(t)$ is a fixed $(m-1)$-plane $\Gamma$.
According to Theorem~\ref{sweep-out-theorem}, 
\[
   \Phi_m[\tilde M(t)] + \Phi_m[\cup_{t\le \tau<0} |\tau|^{-1/2}\Gamma]
\]
is a decreasing function of $t$ for $t<0$.
In particular, for $T < t < 0$,
\begin{align*}
\Phi_m[\tilde M(t)] 
&\le 
\Phi_m[\tilde M(T)] + \Phi_m[\cup_{T\le \tau<0} |\tau|^{-1/2}\Gamma]
\\
&\le
\Phi_m[\tilde M(T)] + \Phi_m[\cup_{-\infty \le \tau<0} |\tau|^{-1/2}\Gamma]
\\
&\le
\Phi_m[\tilde M(T)] + \frac12
\end{align*}
since $Q:=\cup_{-\infty<\tau<0}$ is a halfspace if $0\notin\Gamma$ (in which case $\Phi_m[Q]=1/2$)
or is the $(m-1)$-plane $\Gamma$ if $0\in\Gamma$ (in which case $\Phi_m[Q]=0$).

Likewise, if $\Gamma$ is the union of $\ell$ fixed $(m-1)$-planes, then
\[
\Phi_m[\tilde M(t)] \le \Phi_m[\tilde M(T)] + \frac{\ell}2.
\]

Combining this reasoning with the analysis in the previous sections 
(see, in particular, Theorem~\ref{main-translator-theorem} and Proposition~\ref{mdr-proposition}) gives

\begin{theorem}
Suppose that $M$ is a compact $m$-dimensional translator in $\RR^n$ and that 
\[
 \partial M = (\cup_{i=1}^k\Gamma_i) \cup (\cup_{i=1}^\ell \Gamma'_i),
\]
where each $\Gamma_i$ is contained in a horizontal plane $P_i = \RR^{n-1}\times \{p_i\}$,
and
where each $\Gamma_i'$ is contained in $S_i\times \RR$, where $S_i$ is an $(m-2)$-plane in $\RR^{n-1}$.
Then
\begin{align*}
\entropy(M) 
&\le \sum_{i=1}^k(\mcd(\Gamma_i) + (\entropy (\Gamma_i)) + \frac{\ell}2
\\
&\le \sum_{i=1}^k( \mcd(\Gamma_i) + \mdr(\Gamma_i)) + \frac{\ell}2.
\end{align*}
In particular, if $n=m+1$ and if each $\Gamma_i$ is contained in the boundary of a convex region in $P_i$, then
\[
\entropy(M)
\le
k\left( 1 + \frac{m\omega_m}{\omega_{m-1}}\right) + \frac{\ell}2.
\]
\end{theorem}

\end{document}